\documentclass[12pt,english]{article}
\usepackage[T1]{fontenc}
\usepackage[latin9]{inputenc}
\usepackage{color,graphicx,babel,units,amssymb,amsmath,amsfonts,amsthm}
\usepackage[unicode=true,pdfusetitle,bookmarks=true,bookmarksnumbered=false,bookmarksopen=false,breaklinks=false,pdfborder={0 0 1},backref=false,colorlinks=true]{hyperref}
\hypersetup{linkcolor=blue,citecolor=blue,urlcolor=blue,filecolor=blue}

\makeatletter
\theoremstyle{plain}
\newtheorem{thm}{\protect\theoremname}
  \theoremstyle{plain}
  \newtheorem{prop}[thm]{\protect\propositionname}
  \theoremstyle{definition}
  \newtheorem{defn}[thm]{\protect\definitionname}
  \theoremstyle{remark}
  \newtheorem{rem}[thm]{\protect\remarkname}
  \theoremstyle{plain}
  \newtheorem{lem}[thm]{\protect\lemmaname}

\makeatother

  \providecommand{\definitionname}{Definition}
  \providecommand{\lemmaname}{Lemma}
  \providecommand{\propositionname}{Proposition}
  \providecommand{\remarkname}{Remark}
\providecommand{\theoremname}{Theorem}

\begin{document}

\title{Form Factors for Generalized Grey Brownian Motion}

\author{\textbf{Jos\textbf{{\'e}} Lu\textbf{{\'\i}}s da Silva} \\
 CIMA, University of Madeira, Campus da Penteada,\\
 9020-105 Funchal, Portugal.\\
 Email: luis@uma.pt \and \textbf{Ludwig Streit} \\
 BiBoS, Universitat Bielefeld, Germany,\\
 CIMA, Unversidade da Madeira, Funchal, Portugal\\
 Email: streit@uma.pt}
\maketitle
\begin{abstract}
In this paper we investigate the form factors of paths for a class
of non Gaussian processes. These processes are characterized in terms
of the Mittag-Leffler function. In particular, we obtain a closed
analytic form for the form factors, the Debye function, and can study
their asymptotic decay. 
\end{abstract}
\tableofcontents{}

\section{Introduction}

In recent years fractional Brownian motion and processes related to
fractional dynamics have become an object of intensive study. From
the mathematical point of view these processes in general lack both
the Markov and the semimartingale property, so that many of the classical
methods from stochastic analysis do not apply, making their analysis
more challenging. On the other hand, these processes are capable of
modeling systems with long-range self interaction and memory effects.
In 1992 Schneider introduced the notion of grey Brownian motion \cite{MR1190506}
in order to solve the fractional-time diffusion equation where the
time derivative is a Caputo-Djrbashian derivative of fractional order.
In the 90's Mainardi and coauthors started a systematic study of fractional
differential equations, see for example \cite{Mainardi2010} and references
therein, and introduced the generalized grey Brownian motion (ggBm
for short), and the corresponding fractional-time differential equations
of its density. More recently, Grothaus et al.\ \cite{GJRS14}, developed
an infinite dimensional analysis with respect to (non Gaussian) measures
of Mittag-Leffler type. In addition, in \cite{GJ15} the Green function
of the fractional-time heat equation is constructed by extending the
fractional Feynman-Kac from Schneider \cite{MR1190506}.

Form factors, aka structure factors of geometrical objects, play an
important role in crystallography. More recently they have been put
to use to\ encode geometric features of fractal and/or random objects\footnote{Form factors of random paths play a central role in the theoretical
and experimental analysis of poplymer conformations, see e.g. \cite{Hammouda2008},
\cite{Teraoka2002})}. In particular the form factor $S$ of random paths $\left\{ X(t)|t\in\left[0,n\right]\right\} $
is given by
\begin{equation}
S^{X}(k):=\frac{1}{n^{2}}\int_{0}^{n}dt\int_{0}^{n}ds\,E\left(e^{i\left(k,\left(X(t)-X(s)\right)\right)}\right)
\end{equation}

In this paper we intend to compute and discuss the form factors of
ggBm trajectories; we proceed as follows. In Section\ \ref{sec:ggBm_high_d}
we introduce the background needed later on, first, the functional
setup and the Mittag-Leffler functions (and related functions). We
then define the Mittag-Leffler measures $\mu_{\beta}$, $0<\beta<1$
on the space of vector-valued generalized functions and the generalized
grey Brownian motion ggBm. In Section\ \ref{sec:Form_Factors} we
investigate the form factors for three families of the class ggBm.
In addition, we exhibit their asymptotic decay for large values of
the argument, as well as the equation relating the end-to-end length
and the radius of gyration for the class ggBm.

\section{Generalized grey Brownian motion in arbitrary dimensions}

\label{sec:ggBm_high_d}

\subsection{Prerequisites}

\label{subsec:Prerequisites}Let $d\in$\emph{N} and $L_{d}^{2}$
be the Hilbert space of vector-valued square integrable measurable
functions 
\[
L_{d}^{2}:=L^{2}(\mathbb{R})\otimes\mathbb{R}^{d}.
\]
The space $L_{d}^{2}$ is unitarily isomorphic to a direct sum of
$d$ identical copies of $L^{2}:=L^{2}(\mathbb{R})$, (i.e., the space
of real-valued square integrable measurable functions with Lebesgue
measure). Any element $f\in L_{d}^{2}$ may be written in the form
\begin{equation}
f=(f_{1}\otimes e_{1},\ldots,f_{d}\otimes e_{d}),\label{eq:L2d_element}
\end{equation}
where $f_{i}\in L^{2}(\mathbb{R})$, $i=1,\ldots,d$ and $\{e_{1},\ldots,e_{d}\}$
denotes the canonical basis of $\mathbb{R}^{d}$. The inner product
in $L_{d}^{2}$ is given by 
\[
(f,g)_{0}=\sum_{k=1}^{d}(f_{k},g_{k})_{L^{2}}=\sum_{k=1}^{d}\int_{\mathbb{R}}f_{k}(x)g_{k}(x)\,dx,
\]
where $g=(g_{1}\otimes e_{1},\ldots,g_{d}\otimes e_{d})$, $f_{k}\in L^{2}$,
$k=1,\ldots,d$, $f$ as given in \eqref{eq:L2d_element}. The corresponding
norm in $L_{d}^{2}$ is given by 
\[
|f|_{0}^{2}:=\sum_{k=1}^{d}|f_{k}|_{L^{2}}^{2}=\sum_{k=1}^{d}\int_{\mathbb{R}}f_{k}^{2}(x)\,dx.
\]

As a densely imbedded nuclear space in $L_{d}^{2}$ we choose $S_{d}:=S(\mathbb{R})\otimes\mathbb{R}^{d}$,
where $S(\mathbb{R})$ is the Schwartz test function space. An element
$\varphi\in S_{d}$ has the form 
\begin{equation}
\varphi=(\varphi_{1}\otimes e_{1},\ldots,\varphi_{d}\otimes e_{d}),\label{eq:test_function.}
\end{equation}
where $\varphi_{i}\in S(\mathbb{R})$, $i=1,\ldots,d$. Together with
the dual space $S_{d}^{\prime}:=S^{\prime}(\mathbb{R})\otimes\mathbb{R}^{d}$
we obtain the basic nuclear triple 
\[
S_{d}\subset L_{d}^{2}\subset S_{d}^{\prime}.
\]
The dual pairing between $S_{d}^{\prime}$ and $S_{d}$ is given as
an extension of the scalar product in $L_{d}^{2}$ by 
\[
\langle f,\varphi\rangle_{0}=\sum_{k=1}^{d}(f_{k},\varphi_{k})_{L^{2}},
\]
where $f$ and $\varphi$ as in (\ref{eq:L2d_element}) and (\ref{eq:test_function.}),
respectively. In $S_{d}^{\prime}$ we invoke the Borel $\sigma$-algebra
generated by the cylinder sets.

We define the operator $M_{-}^{\alpha/2}$ on $S(\mathbb{R})$ by 
\[
M_{-}^{\alpha/2}\varphi:=\begin{cases}
K_{\alpha/2}D_{-}^{-\nicefrac{(\alpha-1)}{2}}\varphi, & \alpha\in(0,1),\\
\varphi, & \alpha=1,\\
K_{\alpha/2}I_{-}^{\nicefrac{(\alpha-1)}{2}}\varphi, & \alpha\in(1,2),
\end{cases}
\]
with the normalization constant $K_{\alpha/2}:=\sqrt{\alpha\sin(\nicefrac{\alpha\pi}{2})\Gamma(\alpha)}$
. $D_{-}^{r}$, $I_{-}^{r}$ denote the left-side fractional derivative
and fractional integral of order $r$ in the sense of Riemann-Liouville,
respectively:
\begin{eqnarray*}
(D_{-}^{r}f)(x) & = & \frac{1}{\Gamma(1-r)}\frac{d}{dx}\int_{-\infty}^{x}f(t)(x-t)^{-r}dt,\\
(I_{-}^{r}f)(x) & = & \frac{1}{\Gamma(r)}\int_{x}^{\infty}f(t)(t-x)^{r-1}dt,\hfill x\in\mathbb{R}.
\end{eqnarray*}
We refer to \cite{SKM1993} or \cite{KST2006} for the details on
these operators. It is possible to obtain a larger domain of the operator
$M_{-}^{\nicefrac{\alpha}{2}}$ in order to include the indicator
function $1\!\!1_{[0,t)}$ such that $M_{-}^{\nicefrac{\alpha}{2}}1\!\!1_{[0,t)}\in L^{2}$,
for the details we refer to Appendix A in \cite{GJ15}. We have the
following
\begin{prop}[Corollary 3.5 in \protect\cite{GJ15}]
For all $t,s\ge0$ and all $0<\alpha<2$ it holds that 
\begin{equation}
\big(M_{-}^{\nicefrac{\alpha}{2}}1\!\!1_{[0,t)},M_{-}^{\nicefrac{\alpha}{2}}1\!\!1_{[0,s)}\big)_{L^{2}}=\frac{1}{2}\big(t^{\alpha}+s^{\alpha}-|t-s|^{\alpha}\big).\label{eq:alpha-inner-prod}
\end{equation}
\end{prop}

The Mittag-Leffler function was introduced by G.\ Mittag-Leffler
in a series of papers \cite{Mittag-Leffler1903,Mittag-Leffler1904,Mittag-Leffler1905}.
In Section\ \ref{sec:Form_Factors} we also need its generalization
originally due to R. Agarwal \cite{Agarwal1953}.
\begin{defn}[Mittag-Leffler function]
\label{def:MLf}
\begin{enumerate}
\item For $\beta>0$ the Mittag-Leffler function $E_{\beta}$ (MLf for short)
is defined as an entire function by the following series representation
\begin{equation}
E_{\beta}(z):=\sum_{n=0}^{\infty}\frac{z^{n}}{\Gamma(\beta n+1)},\quad z\in\emph{C},\label{eq:MLf}
\end{equation}
where $\Gamma$ denotes the gamma function.
\item For any $\rho\in\emph{C}$ the generalized Mittag-Leffler function
(gMLf for short) is an entire function defined by its power series
\begin{equation}
E_{\beta,\rho}(z):=\sum_{n=0}^{\infty}\frac{z^{n}}{\Gamma(\beta n+\rho)},\quad z\in\emph{C}.\label{eq:gMLf}
\end{equation}
Note the relation $E_{\beta,1}(z)=E_{\beta}(z)$ and $E_{1}(z)=e^{z}$
for any $z\in\emph{C}$. 
\end{enumerate}
\end{defn}

We have the following asymptotic for the gMLf.
\begin{prop}[{{cf.\ \protect\cite[Section~4.7]{GKMS2014}}}]
Let $0<\beta<2$, $\alpha\in\emph{C}$ and $\delta$ be such that
\[
\frac{\beta\pi}{2}<\delta<\min\{\pi,\beta\pi\}.
\]
Then, for any $m\in\emph{N}$, the following asymptotic formulas hold:
\begin{enumerate}
\item If $|\arg(z)|\leq\delta$, then 
\[
E_{\beta,\alpha}(z)=\frac{1}{\beta}z^{(1-\alpha)/\beta}\exp(z^{1/\beta})-\sum_{n=1}^{m}\frac{z^{-n}}{\Gamma(\alpha-\beta n)}+O(|z|^{-m-1}),\;|z|\rightarrow\infty.
\]
\item If $\delta\le|\arg(z)|\leq\pi$, then 
\begin{equation}
E_{\beta,\alpha}(z)=-\sum_{n=1}^{m}\frac{z^{-n}}{\Gamma(\alpha-\beta n)}+O(|z|^{-m-1}),\quad|z|\rightarrow\infty.\label{eq:asymptotic_MLf}
\end{equation}
\end{enumerate}
\end{prop}

The Euler integral transform of the MLf may be used to compute the
following integral, with $\Re(\alpha),\Re(\beta),\Re(\sigma)>0$,
$\rho\in\emph{C}$ and $\gamma>0$, cf.\ \cite[eq.~(2.2.13)]{Mathai2008}
\begin{equation}
\int_{0}^{1}t^{\rho-1}(1-t)^{\sigma-1}E_{\beta,\alpha}(xt^{\gamma})\,dt=\Gamma(\sigma){}_{2}\Psi_{2}\left(\begin{array}{cc}
(\rho,\gamma), & (1,1)\\
(\alpha,\beta), & (\sigma+\rho,\gamma)
\end{array}\bigg|x\right),\label{eq:Euler_transform-MLf}
\end{equation}
where $_{2}\Psi_{2}$ is the Fox-Wright function (also called generalized
Wright function \cite{Kilbas2002}, \cite[Appendix~F, eq.~(F.2.14)]{GKMS2014}
and \cite{MP07}) given for $x,a_{i},c_{i}\in\emph{C}$ and $b_{i},d_{i}\in\mathbb{R}$
by 
\[
_{2}\Psi_{2}\left(\begin{array}{cc}
(a_{1},b_{1}), & (a_{2},b_{2})\\
(c_{1},d_{1}), & (c_{2},d_{2})
\end{array}\bigg|x\right)=\sum_{n=0}^{\infty}\frac{\Gamma(a_{1}+b_{1}n)\Gamma(a_{2}+b_{2}n)}{\Gamma(c_{1}+d_{1}n)\Gamma(c_{2}+d_{2}n)}\frac{x^{n}}{n!}.
\]
In particular, when $\rho=\alpha$ and $\gamma=\beta$, eq.\ (\ref{eq:Euler_transform-MLf})
yields 
\begin{equation}
\int_{0}^{1}t^{\alpha-1}(1-t)^{\sigma-1}E_{\beta,\alpha}(xt^{\beta})\,dt=\Gamma(\sigma)E_{\beta,\alpha+\sigma}(x).\label{eq:Euler_transform-MLf1}
\end{equation}
Both integrals (\ref{eq:Euler_transform-MLf}) and (\ref{eq:Euler_transform-MLf1})
will be used in Section\ \ref{sec:Form_Factors} below.

\subsection{The Mittag-Leffler measure}

\label{subsec:MLm}The Mittag-Leffler measures $\mu_{\beta}$, $0<\beta\leq1$
are a family of probability measures on $S_{d}^{\prime}$ whose characteristic
functions are given by the Mittag-Leffler functions, see Definition\ \ref{def:MLf}.
Using the Bochner-Minlos theorem, see \cite{GV68}, \cite{H70}, the
following definition makes sense.
\begin{defn}[cf.\ \protect\cite{GJRS14}]
For any $\beta\in(0,1]$ the Mittag-Leffler measure is defined as
the unique probability measure $\mu_{\beta}$ on $S_{d}^{\prime}$
by fixing its characteristic functional 
\begin{equation}
\int_{S_{d}^{\prime}}e^{i\langle w,\varphi\rangle_{0}}\,d\mu_{\beta}(w)=E_{\beta}\left(-\frac{1}{2}|\varphi|_{0}^{2}\right),\quad\varphi\in S_{d}.\label{eq:ch-fc-gnm}
\end{equation}
\end{defn}

\begin{rem}
\label{rem:grey-noise-measure}
\begin{enumerate}
\item The measure $\mu_{\beta}$ is also called grey noise (reference) measure,
cf.\ \cite{GJRS14} and \cite{GJ15}.
\item The range $0<\beta\leq1$ ensures the complete monotonicity of $E_{\beta}(-x)$,
see Pollard \cite{Pollard48}, i.e., $(-1)^{n}E_{\beta}^{(n)}(-x)\ge0$
for all $x\ge0$ and $n\in\emph{N}_{0}:=\{0,1,2,\ldots\}.$ In other
words, this is sufficient to show that 
\[
S_{d}\ni\varphi\mapsto E_{\beta}\left(-\frac{1}{2}|\varphi|_{0}^{2}\right)\in\mathbb{R}
\]
is a characteristic function in $S_{d}$. 
\end{enumerate}
\end{rem}

We consider the complex Hilbert space of square integrable measurable
functions defined on $S_{d}^{\prime}$, $L^{2}(\mu_{\beta}):=L^{2}(S_{d}^{\prime},\mathcal{B},\mu_{\beta}),$
with scalar product defined by 
\[
(\!(F,G)\!)_{L^{2}(\mu_{\beta})}:=\int_{S_{d}^{\prime}}F(w)\bar{G}(w)\,d\mu_{\beta}(w),\quad F,G\in L^{2}(\mu_{\beta}).
\]
The corresponding norm is denoted by $\Vert\cdot\Vert_{L^{2}(\mu_{\beta})}$.
It follows from (\ref{eq:ch-fc-gnm}) that all moments of $\mu_{\beta}$
exist and we have
\begin{lem}
\label{lem:gnm}For any $\varphi\in S_{d}$ and $n\in\emph{N}_{0}$
we have 
\begin{eqnarray*}
\int_{S_{d}^{\prime}}\langle w,\varphi\rangle_{0}^{2n+1}\,d\mu_{\beta}(w) & = & 0,\\
\int_{S_{d}^{\prime}}\langle w,\varphi\rangle_{0}^{2n}\,d\mu_{\beta}(w) & = & \frac{(2n)!}{2^{n}\Gamma(\beta n+1)}|\varphi|_{0}^{2n}.
\end{eqnarray*}
In particular, $\|\langle\cdot,\varphi\rangle\|_{L^{2}(\mu_{\beta})}^{2}=\frac{1}{\Gamma(\beta+1)}|\varphi|_{0}^{2}$
and by polarization for any $\varphi,\psi\in S_{d}$ we obtain 
\[
\int_{S_{d}^{\prime}}\langle w,\varphi\rangle_{0}\langle w,\psi\rangle_{0}\,d\mu_{\beta}(w)=\frac{1}{\Gamma(\beta+1)}\langle\varphi,\psi\rangle_{0}.
\]
\end{lem}

\subsection{Generalized grey Brownian motion}

\label{subsec:ggBm}For any test function $\varphi\in S_{d}$ we define
the random variable 
\[
X^{\beta}(\varphi):S_{d}^{\prime}\longrightarrow\mathbb{R}^{d},\;w\mapsto X^{\beta}(\varphi,w):=\big(\langle w_{1},\varphi_{1}\rangle,\ldots,\langle w_{d},\varphi_{d}\rangle\big).
\]
The random variable $X^{\beta}(\varphi)$ has the following properties
which are a consequence of Lemma\ \ref{lem:gnm} and the characteristic
function of $\mu_{\beta}$ given in (\ref{eq:ch-fc-gnm}).
\begin{prop}
Let $\varphi,\psi\in S_{d}$, $k\in\mathbb{R}^{d}$ be given. Then
\begin{enumerate}
\item The characteristic function of $X^{\beta}(\varphi)$ is given by 
\begin{equation}
\emph{E}\big(e^{i(k,X^{\beta}(\varphi))}\big)=E_{\beta}\left(-\frac{1}{2}\sum_{j=1}^{d}k_{j}^{2}|\varphi_{j}|_{L^{2}}^{2}\right).\label{eq:characteristic-coord-proc}
\end{equation}
\item The characteristic function of the random variable $X^{\beta}(\varphi)-X^{\beta}(\psi)$
is 
\begin{equation}
\emph{E}\big(e^{i(k,X^{\beta}(\varphi)-X^{\beta}(\psi))}\big)=E_{\beta}\left(-\frac{1}{2}\sum_{i=1}^{d}k_{j}^{2}|\varphi_{j}-\psi_{j}|_{L^{2}}^{2}\right).\label{eq:CF_increment}
\end{equation}
\item The expectation of the $X^{\beta}(\varphi)$ is zero and 
\begin{equation}
\Vert X^{\beta}(\varphi)\Vert_{L^{2}(\mu_{\beta})}^{2}=\frac{1}{\Gamma(\beta+1)}|\varphi|_{0}^{2}.\label{eq:variance.cood-proc}
\end{equation}
\item The moments of $X^{\beta}(\varphi)$ are given by 
\begin{eqnarray*}
\int_{S_{d}^{\prime}}\big|X^{\beta}(\varphi,w)\big|^{2n+1}\,d\mu_{\beta}(w) & = & 0,\\
\int_{S_{d}^{\prime}}\big|X^{\beta}(\varphi,w)\big|^{2n}\,d\mu_{\beta}(w) & = & \frac{(2n)!}{2^{n}\Gamma(\beta n+1)}|\varphi|_{0}^{2n}.
\end{eqnarray*}
\end{enumerate}
\end{prop}

\begin{rem}
$X^{\beta}$ is a \textquotedbl{}generalized process\textquotedbl{},
with white noise as a special case: $\mu_{1}$ is the Gaussian white
noise measure \cite{GV68}. 
\end{rem}

The property (\ref{eq:variance.cood-proc}) of $X^{\beta}(\varphi)$
gives the possibility to extend the definition of $X^{\beta}$ to
any element in $L_{d}^{2}$, in fact, if $f\in L_{d}^{2}$, then there
exists a sequence $(\varphi_{k})_{k=1}^{\infty}\subset S_{d}$ such
that $\varphi_{k}\longrightarrow f$, $k\rightarrow\infty$ in the
norm of $L_{d}^{2}$. Hence, the sequence $\big(X^{\beta}(\varphi_{k})\big)_{k=1}^{\infty}\subset L^{2}(\mu_{\beta})$
forms a Cauchy sequence which converges to an element denoted by $X^{\beta}(f)\in L^{2}(\mu_{\beta})$.

We define $1\!\!1_{[0,t)}\in L_{d}^{2}$, $t\ge0$, by 
\[
1\!\!1_{[0,t)}:=(1\!\!1_{[0,t)}\otimes e_{1},\ldots,1\!\!1_{[0,t)}\otimes e_{d})
\]
and consider the process $X^{\beta}(1\!\!1_{[0,t)})\in L^{2}(\mu_{\beta})$
such that the following definition makes sense.
\begin{defn}
For any $0<\alpha<2$ we define the process
\begin{eqnarray}
S_{d}^{\prime}\ni w\mapsto B^{\beta,\alpha}(t,w) & := & \left(\langle w,(M_{-}^{\alpha/2}1\!\!1_{[0,t)})\otimes e_{1}\rangle,\ldots,\langle w,(M_{-}^{\alpha/2}1\!\!1_{[0,t)})\otimes e_{d}\rangle\right)\nonumber \\
 & = & \left(\langle w_{1},M_{-}^{\alpha/2}1\!\!1_{[0,t)}\rangle,\ldots,\langle w_{d},M_{-}^{\alpha/2}1\!\!1_{[0,t)}\rangle\right),\;t>0\label{eq:ggBm}
\end{eqnarray}
as an element in $L^{2}(\mu_{\beta})$ and call this process $d$-dimensional
generalized grey Brownian motion (ggBm), in short
\[
B^{\beta,\alpha}(t)=X^{\beta}\left(M_{-}^{\alpha/2}1\!\!1_{[0,t)}\right).
\]
\end{defn}

\begin{rem}
\label{rem:existence_ggBm}
\begin{enumerate}
\item The $d$-dimensional ggBm $B^{\beta,\alpha}$ exist as a $L^{2}(\mu_{\beta})$-limit
and hence the map $S'_{d}\ni\omega\mapsto\langle\omega,M_{-}^{\alpha/2}1\!\!1_{[0,t)}\rangle$
yields a version of ggBm, $\mu_{\beta}$-a.s., but not in the pathwise
sense. 
\item For any fixed $0<\alpha<2$ one can show by the Kolmogorov-Centsov
continuity theorem that the paths of the process are $\mu_{\beta}$-a.s.\ continuous,
cf.\ \cite[Prop.~3.8]{GJ15}.
\item Below we manly deal with expectation of functions of ggBm, therefore
the version of ggBm defined above is sufficient.
\end{enumerate}
\end{rem}

\begin{prop}
For any $0<\alpha<2$, the process $B^{\beta,\alpha}:=\{B^{\beta,\alpha}(t),\;t\geq0\}$,
is $\alpha/2$ self-similar with stationary increments. 
\end{prop}

\begin{proof}
Given $k=(k_{1},k_{2},\ldots,k_{n})\in\mathbb{R}^{n}$, we have to
show that for any $0<t_{1}<t_{2}<\ldots<t_{n}$ and $a>0$: 
\begin{equation}
\emph{E}\left(\exp\left(i\left\langle \cdot,\sum_{j=1}^{n}k_{j}M_{-}^{\nicefrac{\alpha}{2}}1\!\!1_{[0,at_{j})}\right\rangle \right)\right)=\emph{E}\left(\exp\left(ia^{\nicefrac{\alpha}{2}}\left\langle \cdot,\sum_{j=1}^{n}k_{j}M_{-}^{\nicefrac{\alpha}{2}}1\!\!1_{[0,t_{j})}\right\rangle \right)\right).\label{eq:equality-ch-fc}
\end{equation}
It follows from (\ref{eq:characteristic-coord-proc}) that eq.\ (\ref{eq:equality-ch-fc})
is equivalent to 
\[
E_{\beta}\left(-\frac{1}{2}\left\vert \sum_{j=1}^{n}k_{j}M_{-}^{\nicefrac{\alpha}{2}}1\!\!1_{[0,at_{j})}\right\vert _{L^{2}}^{2}\right)=E_{\beta}\left(-\frac{1}{2}\left\vert a^{\nicefrac{\alpha}{2}}\sum_{j=1}^{n}k_{j}M_{-}^{\nicefrac{\alpha}{2}}1\!\!1_{[0,t_{j})}\right\vert _{L^{2}}^{2}\right).
\]
Because of the complete monotonicity of $E_{\beta}$, the above equality
reduces to 
\[
\left\vert \sum_{j=1}^{n}k_{j}M_{-}^{\nicefrac{\alpha}{2}}1\!\!1_{[0,at_{j})}\right\vert _{L^{2}}^{2}=a^{\alpha}\left\vert \sum_{j=1}^{n}k_{j}M_{-}^{\nicefrac{\alpha}{2}}1\!\!1_{[0,t_{j})}\right\vert _{L^{2}}^{2}
\]
which is easy to show, taking into account (\ref{eq:alpha-inner-prod}).
A similar procedure may be applied in order to prove the stationarity
of the increments. Hence, for any $h\geq0$, we have to show that
\[
\emph{E}\left(\exp\left(i\sum_{j=1}^{n}k_{j}\big(B^{\beta,\alpha}(t_{j}+h)-B^{\beta,\alpha}(h)\big)\right)\right)=\emph{E}\left(\exp\left(i\sum_{j=1}^{n}k_{j}B^{\beta,\alpha}(t_{j})\right)\right).
\]
The above procedure reduces this equality to check the following 
\[
\left\vert \sum_{j=1}^{n}k_{j}M_{-}^{\nicefrac{\alpha}{2}}1\!\!1_{[h,t_{j}+h)}\right\vert _{L^{2}}^{2}=\left\vert \sum_{j=1}^{n}k_{j}M_{-}^{\nicefrac{\alpha}{2}}1\!\!1_{[0,t_{j})}\right\vert _{L^{2}}^{2}
\]
which is verified as before. 
\end{proof}
\begin{rem}
\label{rem:self-similar}The family $\{B^{\beta,\alpha}(t),\;t\geq0,\,\beta\in(0,1],\,\alpha\in(0,2)\}$
forms a class of $\alpha/2$-self-similar processes with stationary
increments which includes:
\begin{enumerate}
\item for $\beta=\alpha=1$, the process $\{B^{1,1}(t),\;t\geq0\}$, standard
$d$-dimensional Bm.
\item for $\beta=1$ and $0<\alpha<2$, $\{B^{1,\alpha}(t),\;t\geq0\}$,
$d$-dimensional fBm with Hurst parameter $\alpha/2$.
\item for $\alpha=1$, $\{B^{\beta,1}(t),\;t\geq0\}$ a $1/2$-self-similar
non Gaussian process with 
\begin{equation}
\emph{E}\left(e^{i\left(k,B^{\beta,1}(t)\right)}\right)=E_{\beta}\left(-\frac{|k|^{2}}{2}t\right),\quad k\in\mathbb{R}^{d}.\label{eq:ch-fc-1/2sssi}
\end{equation}
\item for $0<\alpha=\beta<1$, the process $\{B^{\beta}(t):=B^{\beta,\beta}(t),\;t\geq0\}$,
$\beta/2$ self-similar and called $d$-dimensional grey Brownian
motion (gBm for short). The characteristic function of\ $B^{\beta}(t)$
is given by 
\begin{equation}
\emph{E}\left(e^{i\left(k,B^{\beta}(t)\right)}\right)=E_{\beta}\left(-\frac{|k|^{2}}{2}t^{\beta}\right),\quad k\in\mathbb{R}^{d}.\label{eq:ch-fc-gBm}
\end{equation}
For $d=1$, gBm was introduced by W.\ Schneider in \cite{Schneider90,MR1190506}.
\item For other choices of the parameters $\beta$ and $\alpha$ we obtain,
in general, non Gaussian processes. 
\end{enumerate}
\end{rem}

\section{Form factors for different classes of ggBm}

\label{sec:Form_Factors}In this section we investigate the form factors
of ggBm and two particular cases introduced in Subsection\ \ref{subsec:ggBm},
cf.\ Remark~\ref{rem:self-similar}. To begin with we note that,
given a $d$-dimensional stochastic process $X$, the form factor
associated to $X$ is the function defined by 
\begin{equation}
S^{X}(k):=\frac{1}{n^{2}}\int_{0}^{n}dt\int_{0}^{n}ds\,\emph{E}\big(e^{i(k,X(t)-X(s))}\big),\quad k\in\mathbb{R}^{d},\;n\in\emph{N}\label{eq:form_factor_general}
\end{equation}
which, in case $X$ is $\nu$-self-similar, reduces to 
\begin{equation}
S^{X}(k)=\int_{0}^{1}dt\int_{0}^{1}ds\,\emph{E}\big(e^{in^{\nu}(k,X(t)-X(s))}\big).\label{eq:form_factor_general1}
\end{equation}
This function encodes in particular, to lowest order in $k$, the
\emph{root-mean-square radius of gyration} (or simply \emph{radius
of gyration}) of $X$, defined by 
\[
\left(R_{g}^{X}\right)^{2}:=\frac{1}{2}\frac{1}{n^{2}}\int_{0}^{n}dt\int_{0}^{n}ds\,\emph{E}\left(\big|X(t)-X(s)\big|^{2}\right)
\]
play an important role in the study of random path conformations.
For $\nu$-self-similar $d$-dimensional processes $X$, the radius
of gyration obeys a general relation to the mean square end-to-end
length 
\[
(R_{e}^{X})^{2}:=\,\emph{E}\left(\big|X(n)-X(0)\big|^{2}\right)
\]
of the paths, namely 
\begin{equation}
\frac{(R_{e}^{X})^{2}}{(2\nu+1)(2\nu+2)}=(R_{g}^{X})^{2}.\label{scale}
\end{equation}
In particular, for a $d$-dimensional Bm $B$ ($\nu=1/2$) we have
the celebrated equality, see for example \cite[Section~2.4]{Teraoka2002}
\begin{equation}
\frac{(R_{e}^{B})^{2}}{6}=(R_{g}^{B})^{2}.\label{eq:end2end_gyrationBm}
\end{equation}
In the following subsections we compute explicitly these notions for
the different classes of ggBm.

\subsection{The form factor for ggBm $B^{\beta,\alpha}$}

The most general case of the two parameter family $B^{\beta,\alpha}=\{B^{\beta,\alpha}(t),\,t\geq0\}$
with $0<\beta<1$ and $0<\alpha<2$ is now prepared to be investigated.
It corresponds to a $d$-dimensional ggBm process which is $\alpha/2$-self-similar
with stationary increments, cf.\ Section\ \ref{subsec:ggBm}. The
characteristic function of the increment of ggBm is given by 
\begin{equation}
\emph{E}\left(e^{i(k,B^{\beta,\alpha}(t)-B^{\beta,\alpha}(s))}\right)=E_{\beta}\left(-\frac{|k|^{2}}{2}|t-s|^{\alpha}\right).\label{nome}
\end{equation}
Denote by $S^{\beta,\alpha}:=S^{B^{\beta,\alpha}}$ the structure
factors for the ggBm process $B^{\beta,\alpha}$. Then, according
to (\ref{eq:form_factor_general1}) $S^{\beta,\alpha}$ is given by
\[
S^{\beta,\alpha}(k):=\int_{0}^{1}dt\int_{0}^{1}ds\,\emph{E}\big(e^{in^{1/2}(k,B^{\beta,\alpha}(t)-B^{\beta,\alpha}(s))}\big),\quad k\in\mathbb{R}^{d},\;n\in\emph{N}
\]
and it follows from (\ref{nome}) that $S^{\beta,\alpha}$ is equal
to 
\[
S^{\beta,\alpha}(k)=\int_{0}^{1}dt\int_{0}^{1}ds\,E_{\beta}\left(-\frac{n^{\alpha}|k|^{2}}{2}|t-s|^{\alpha}\right)=2\int_{0}^{1}dt\int_{0}^{t}ds\,E_{\beta}\left(-\frac{n^{\alpha}|k|^{2}}{2}|t-s|^{\alpha}\right).
\]
A change of variables $\tau=t-s$ reduces $S^{\beta,\alpha}$ to the
following integral 
\[
S^{\beta,\alpha}(k)=2\int_{0}^{1}d\tau\,(1-\tau)E_{\beta}\left(-\frac{n^{\alpha}|k|^{2}}{2}\tau^{\alpha}\right).
\]
Using eq.\ (\ref{eq:Euler_transform-MLf}) the above integral can
be computed and the explicit expression for $S^{\beta,\alpha}$ is
given in terms of the Fox-Wright function 
\[
S^{\beta,\alpha}(k)=2\,_{2}\Psi_{2}\left(\begin{array}{cc}
(1,\alpha), & (1,1)\\
(1,\beta), & (3,\alpha)
\end{array}\bigg|-\frac{n^{\alpha}|k|^{2}}{2}\right)
\]
or by its Taylor series as 
\begin{eqnarray*}
S^{\beta,\alpha}(k) & = & 2\sum_{j=0}^{\infty}\frac{\Gamma(\alpha j+1)\Gamma(j+1)}{\Gamma(\beta j+1)\Gamma(\alpha j+3)}\frac{(-\frac{1}{2}|k|^{2}n^{\alpha})^{j}}{j!}\\
 & = & 2\sum_{j=0}^{\infty}\frac{1}{\Gamma(\beta j+1)(\alpha j+1)(\alpha j+2)}\left(-\frac{n^{\alpha}|k|^{2}}{2}\right)^{j}.
\end{eqnarray*}
For fixed $\beta,\alpha$ the form factor depends only on $y^{2}=\frac{n^{\alpha}|k|^{2}}{2}$
via the so-called Debye function $f_{D}:$
\[
S^{\beta,\alpha}(k)=:f_{D}(y;\beta,\alpha),
\]

For the process $B^{\beta,\alpha}$ we thus have 
\begin{equation}
f_{D}(y;\beta,\alpha)=2\sum_{j=0}^{\infty}\frac{1}{\Gamma(\beta j+1)(\alpha j+1)(\alpha j+2)}\left(-y^{2}\right)^{j}.
\end{equation}
\begin{rem}
\label{rem:limits_ggBm}Among possible limit cases of the Debye function
$f_{D}$ we point out:
\begin{enumerate}
\item for $\beta=\alpha=1$, clearly we obtain the Debye function $f_{D}(\cdot;1,1)$
for Bm, see (\ref{eq:Debye_Bm}).
\item for any $0<\alpha<2$ fixed, the limit $\beta\rightarrow1$ is 
\begin{eqnarray}
f_{D}(y;1,\alpha) & = & 2\sum_{n=0}^{\infty}\frac{(-y^{2})^{n}}{n!(\alpha n+2)(\alpha n+1)}\nonumber \\
 & = & \frac{2(y^{2})^{-1/\alpha}}{\alpha}\gamma\left(\frac{1}{\alpha},y^{2}\right)-\frac{2(y^{2})^{-1/\alpha}}{\alpha}\gamma\left(\frac{2}{\alpha},y^{2}\right),\label{eq:Debye_fc_fBm}
\end{eqnarray}
which coincides with the Debye function for fBm with Hurst parameter
$H=\alpha/2,$ and $\gamma$ is the incomplete gamma function. 
\end{enumerate}
In Figure\ \ref{fig:Debye-fc-ggBm} we present the plots (linear
scale in left and loglog scale right) of the Debye function $f_{D}$
for different values of $\beta$ and $\alpha$. The asymptotic in
general is harder to obtain for the Fox-Wright function. In \cite{Braaksma1936}
and \cite{Kilbas2002} there are studies on the asymptotic for certain
classes of special functions which include the Fox-Wright function. 
\end{rem}

\begin{figure}
\begin{centering}
\includegraphics[scale=0.45]{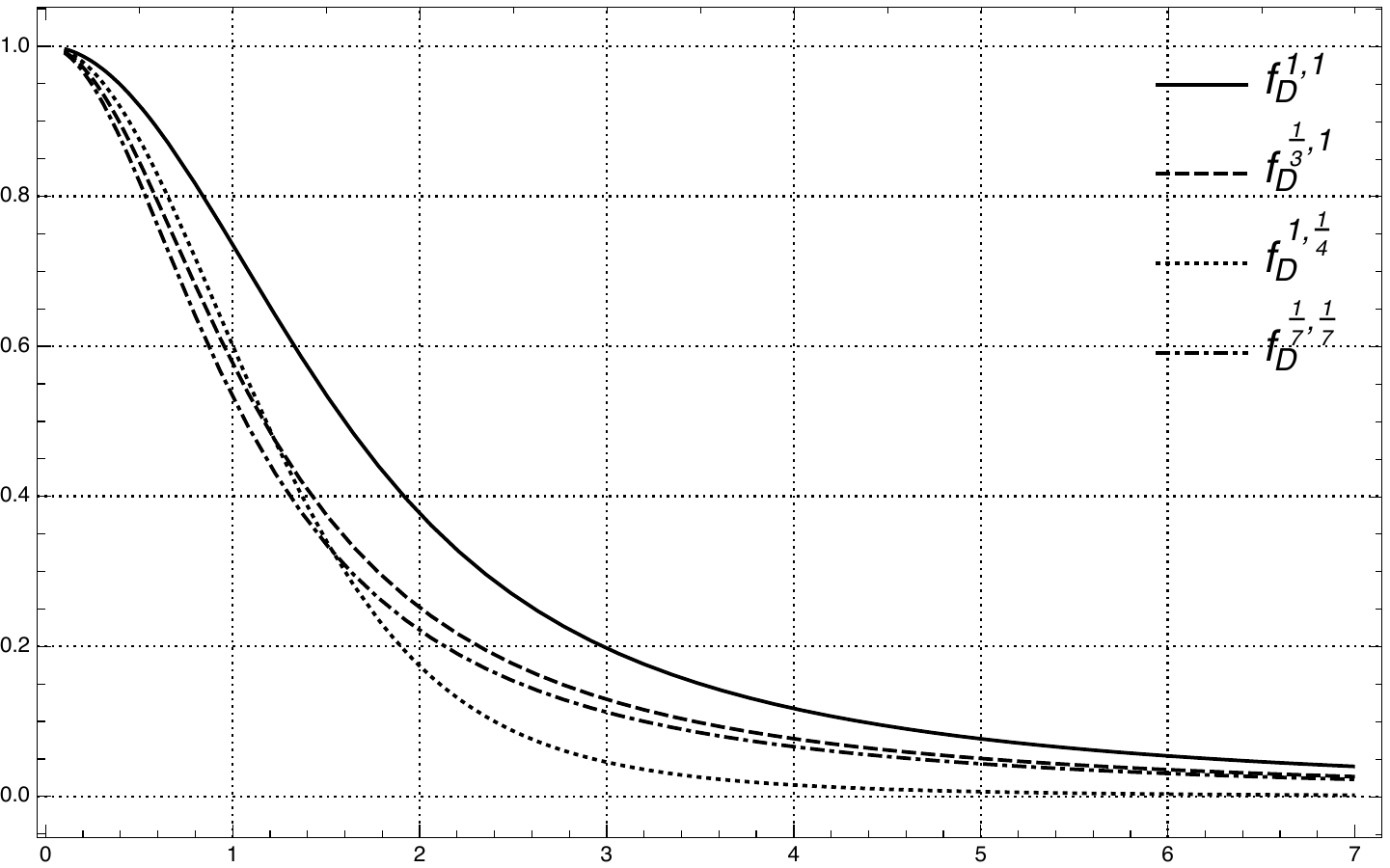}\hfill{}\includegraphics[scale=0.45]{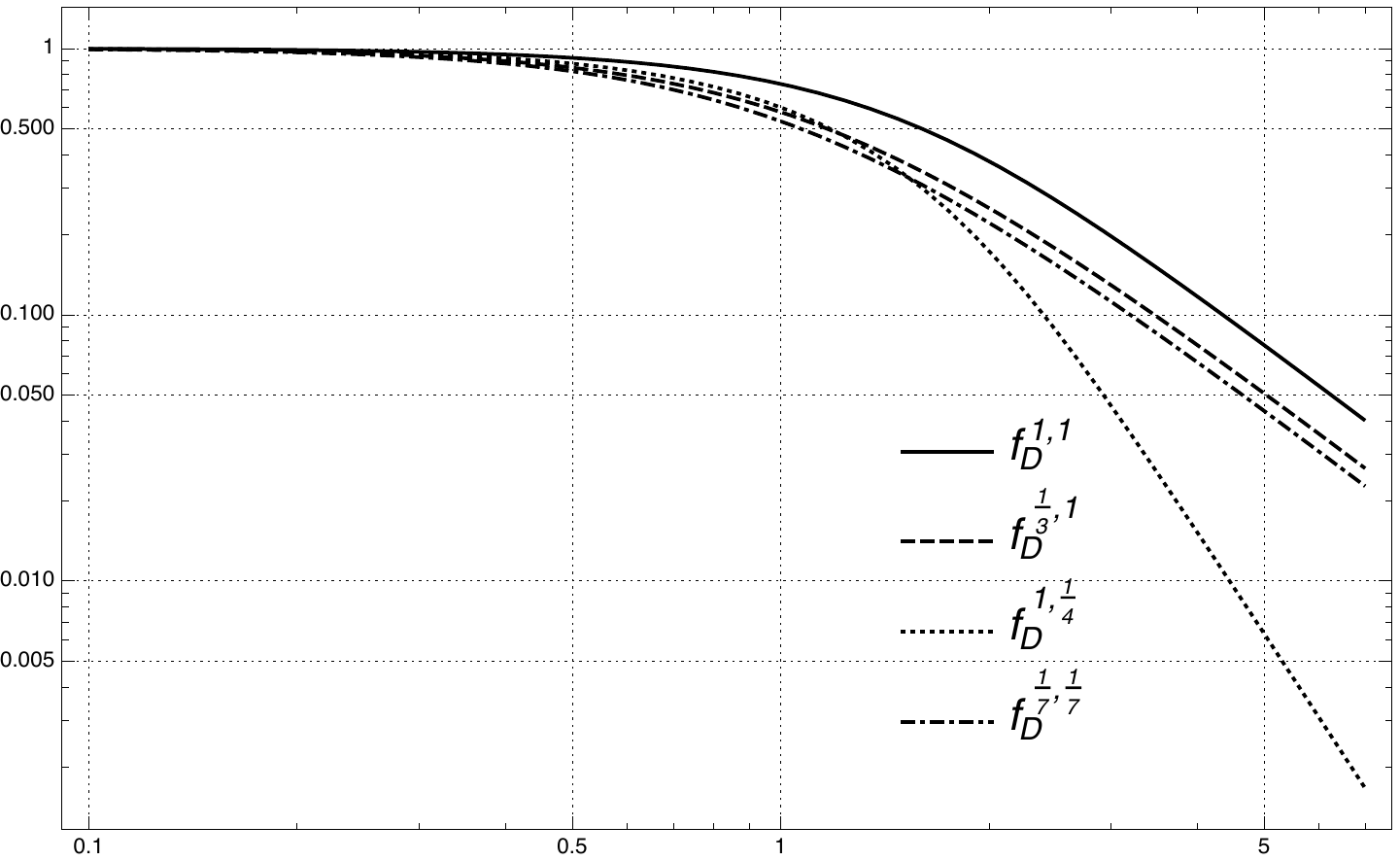} 
\par\end{centering}
\caption{The Debye function $f_{D}^{\beta,\alpha}=f_{D}(\cdot;\beta,\alpha)$.
Left linear scale; right log-log scale.}
\label{fig:Debye-fc-ggBm} 
\end{figure}

The radius of gyration for ggBm is obtained by expanding the form
factor to lowest order
\[
\left(R_{g}^{\beta,\alpha}\right)^{2}\ =\frac{n^{\alpha}}{\Gamma(\beta+1)(\alpha+1)(\alpha+2)}\ 
\]
and with
\[
\left(R_{e}^{\beta,\alpha}\right)^{2}=E\left(\left(B^{\beta,\alpha}(n)\right)^{2}\right)=\frac{n^{\alpha}}{\Gamma(\beta+1)}
\]
we have

\[
\frac{\left(R_{e}^{\beta,\alpha}\right)^{2}}{(\alpha+1)(\alpha+2)}=\left(R_{g}^{\beta,\alpha}\right)^{2}.
\]

\subsection{The form factor for gBm $B^{\beta}$}

For $\alpha=\beta$ the above specializes to gBm, with the Debye function
\[
f_{D}(y;\beta,\beta)=2\sum_{j=0}^{\infty}\frac{1}{\Gamma(\beta j+1)(\beta j+1)(\beta j+2)}\left(-y^{2}\right)^{j}=2E_{\beta,3}(-y^{2}).
\]
\begin{rem}
\label{rem:limits_gBm}The two limit cases ($\beta\rightarrow1$ and
$\beta\rightarrow0$) of the Debye function $f_{D}(\cdot;\beta,\beta)$
are then as follows.
\begin{enumerate}
\item For $\beta\rightarrow1$ we obtain 
\begin{equation}
f_{D}(y;1,1)=2\sum_{j=0}^{\infty}\frac{(-y^{2})^{j}}{(j+2)!}=\frac{2}{y^{4}}\left(e^{-y^{2}}-1+y^{2}\right),\label{eq:Debye_gBm_beta_1}
\end{equation}
i.e., the Debye function for Bm, cf.\ \cite[Ch.~28]{Hammouda2008}.
\item The other limit $\beta\rightarrow0$ is given by 
\begin{equation}
\lim_{\beta\rightarrow0}f_{D}(y;\beta,\beta)=\frac{1}{1+y^{2}},\quad|y|<1.\label{eq:Debye_gBm_beta_zero}
\end{equation}
\end{enumerate}
\end{rem}

\begin{rem}
In Figure\ \ref{fig:Debye-fc-gBm} we show the plots of the Debye
function $f_{D}(\cdot;\beta,\beta)$ for different values of $\beta$.
\begin{enumerate}
\item For large values of $y$ the function $f_{D}(\cdot;1,1)$ falls off
as $2y^{-2}$ (Bm case), right plot upper continuous line.
\item On the other hand, all the other cases $0<\beta<1$, the function
$f_{D}(\cdot;\beta,\beta)$ falls according to the asymptotic of the
Mittag-Leffler function $E_{\beta,3}$ given in (\ref{eq:asymptotic_MLf}).
In that case, we have 
\[
f_{D}(y;\beta,\beta)=2\sum_{j=1}^{N-1}\frac{(-y^{2})^{-j}}{\Gamma(3-\beta j)}\sim\frac{2}{\Gamma(3-\beta)}y^{-2}.
\]
In terms of the right plot in Figure\ \ref{fig:Debye-fc-gBm}, this
corresponds to the lower lines. 
\end{enumerate}
\end{rem}

\begin{figure}
\begin{centering}
\includegraphics[scale=0.45]{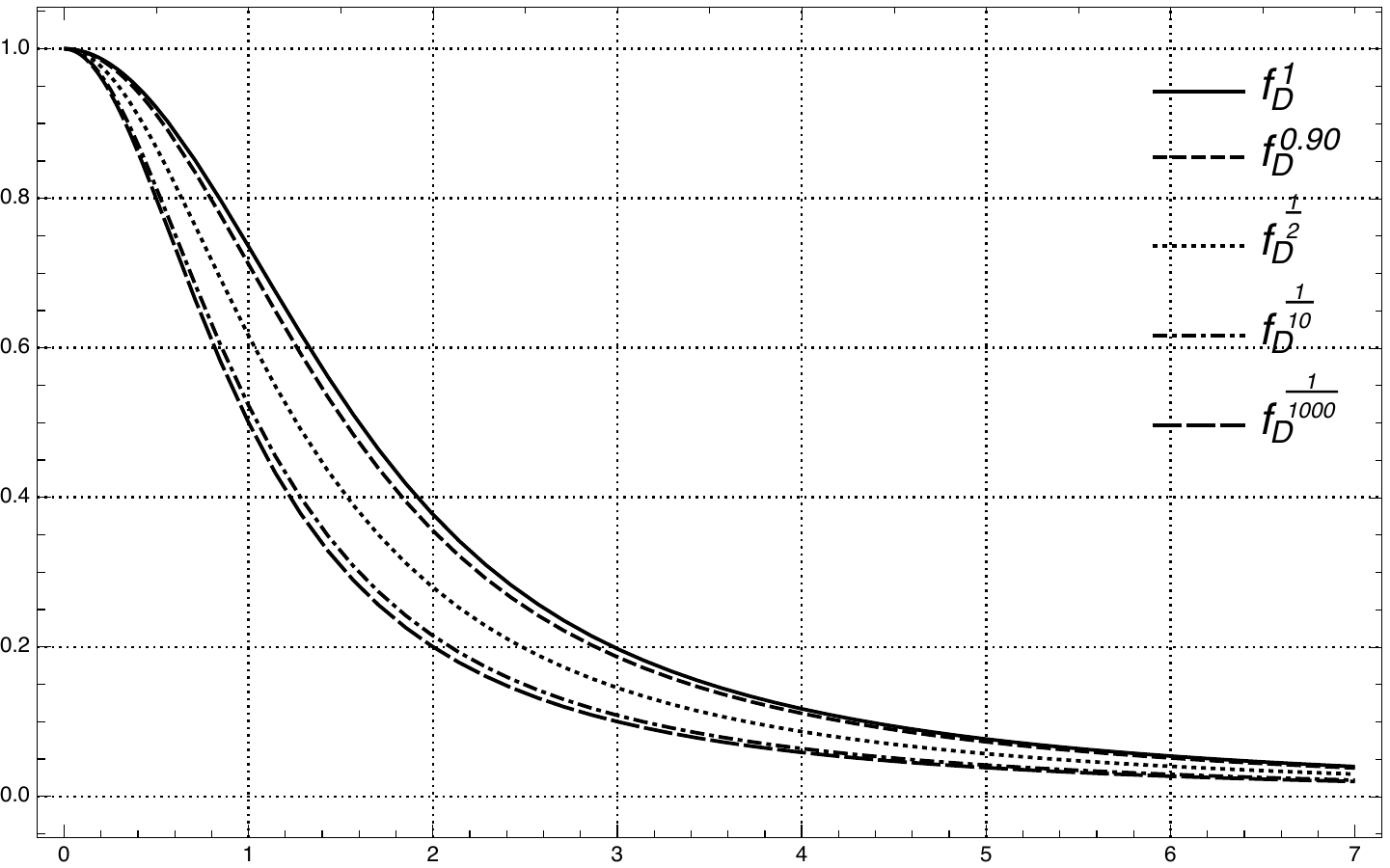}\hfill{}\includegraphics[scale=0.45]{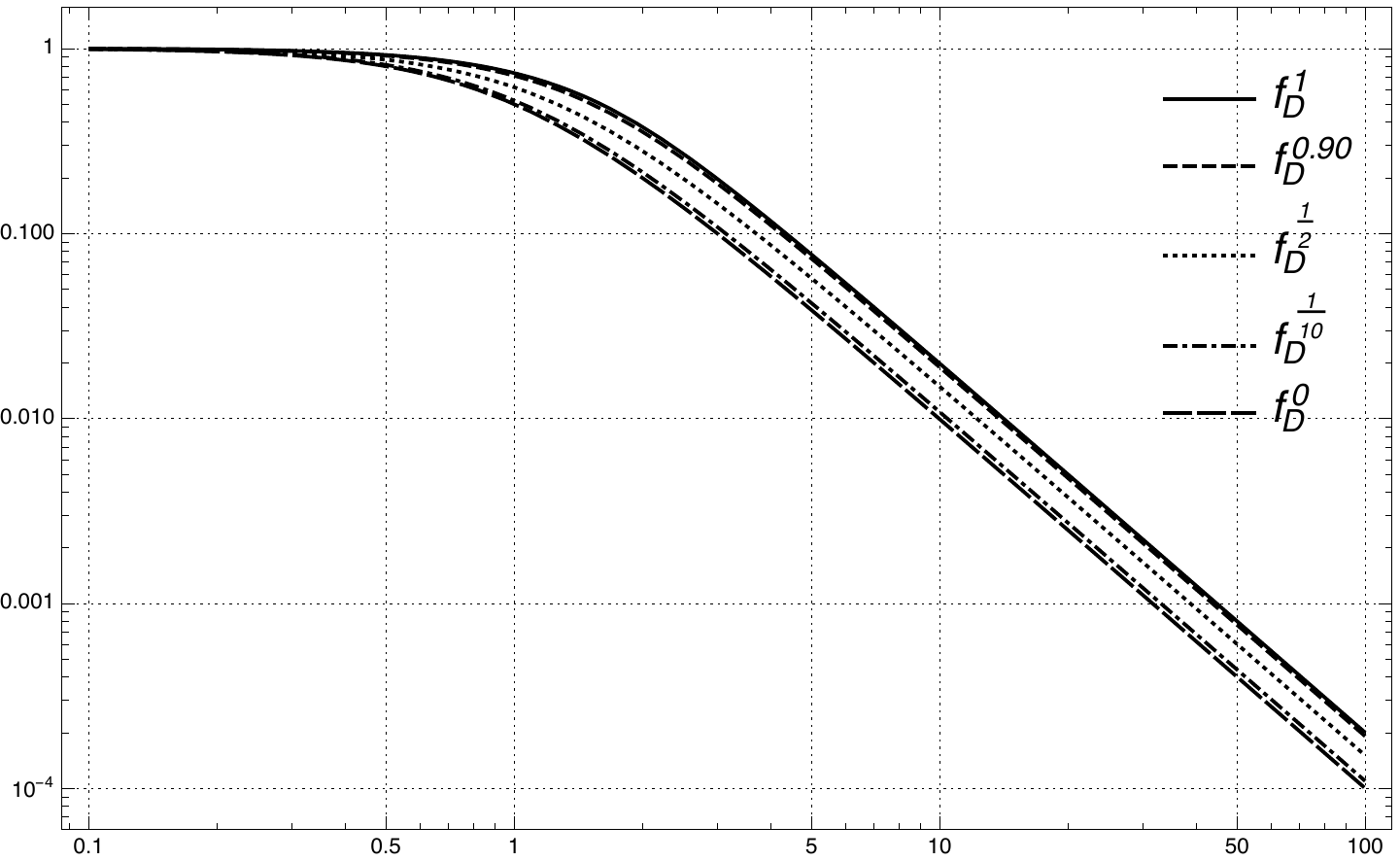} 
\par\end{centering}
\caption{The Debye function $f_{D}^{\beta}=f_{D}(\cdot;\beta,\beta)$ for gBm.
Left linear scale; right logaritmic scale.}
\label{fig:Debye-fc-gBm} 
\end{figure}

The gBm is a $\beta/2-$self-similar process,\ so the relation between
the end-to-end length and the radius of gyration for gBm is given
by 
\begin{equation}
\frac{\big(R_{e}^{\beta}\big)^{2}}{(\beta+1)(\beta+2)}=\big(R_{g}^{\beta}\big)^{2}.\label{eq:e2e_gyration_gBm}
\end{equation}

\subsection{The form factor for $B^{\beta,1}$}

In this case we have 
\begin{eqnarray}
f_{D}(y;\beta,1) & = & 2\sum_{j=0}^{\infty}\frac{(-y^{2})^{j}}{\Gamma(\beta j+1)(j+1)(j+2)}\nonumber \\
 & = & 2\sum_{j=0}^{\infty}\frac{(-y^{2})^{j}}{\Gamma(\beta j+1)(j+1)}-2\sum_{j=0}^{\infty}\frac{(-y^{2})^{j}}{\Gamma(\beta j+1)(j+2)}.\label{eq:Debye_fc_beta1}
\end{eqnarray}
\begin{rem}
\label{rem:limits_Debye_beta1}The two limit cases ($\beta\rightarrow1$
and $\beta\rightarrow0$) of the Debye function $f_{D}(\cdot;\beta,1)$
follows from (\ref{eq:Debye_fc_beta1}). Namely
\begin{enumerate}
\item For $\beta\rightarrow1$ we have 
\begin{equation}
f_{D}(y;1,1)\ =\frac{2}{y^{4}}\left(e^{-y^{2}}-1+y^{2}\right)\label{eq:Debye_Bm}
\end{equation}
which coincides with the Debye function for Bm.
\item The other limit $\beta\rightarrow0$ is obtained as 
\begin{eqnarray}
\lim_{\beta\rightarrow0}f_{D}(y;\beta,1) & = & 2\sum_{j=0}^{\infty}\frac{(-y^{2})^{j}}{(j+1)(j+2)}\nonumber \\
 & = & \frac{2}{y^{4}}\left(-y^{2}+(1+y^{2})\ln(1+y^{2})\right).\label{eq:Debye_zero1}
\end{eqnarray}
\end{enumerate}
\end{rem}

\begin{rem}
In Figure\ \ref{fig:Debye-fc-beta1} we present the plots of the
Debye function$f_{D}(y;\beta,1)$ for different values of $\beta$.
\begin{enumerate}
\item As it follows from (\ref{eq:Debye_Bm}) for large values of $y$ the
function $f_{D}(\cdot;1,1)$ falls as $2y^{-2}$, right plot lower
dash-dot line.
\item The other limit case $\beta\rightarrow0$, eq.\ (\ref{eq:Debye_zero1})
implies that $\lim_{\beta\rightarrow0}f_{D}(y;\beta,1)$ falls as
$\left(-2+4\ln(y)\right)y^{-2}$, right plot upper continuous line.
\item In the intermediate cases $\ f_{D}(y;\beta,1)$ decays as $\left(k_{1}+k_{2}\ln(y)\right)y^{-2}$
where $k_{1},k_{2}$ are constants. As an example, for $\beta=\frac{1}{3}$,
\[
f_{D}(y;1/3,\beta)\approx\left(-0.827976+2.95395\ln(y)\right)y^{-2},\;y\rightarrow\infty,
\]
dash and dotted lines in the right hand plot. 
\end{enumerate}
\end{rem}

\begin{figure}
\begin{centering}
\includegraphics[scale=0.45]{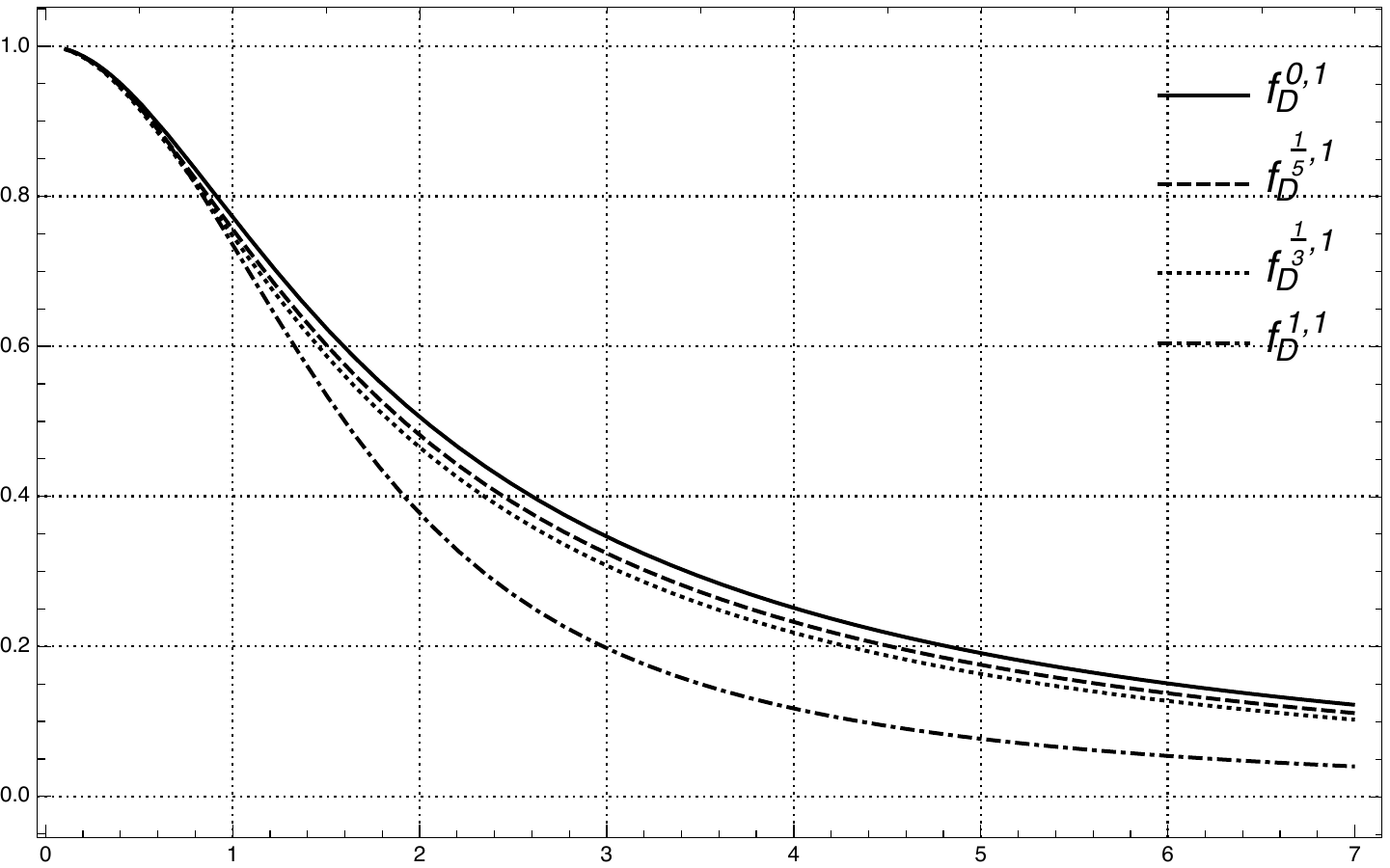}\hfill{}\includegraphics[scale=0.45]{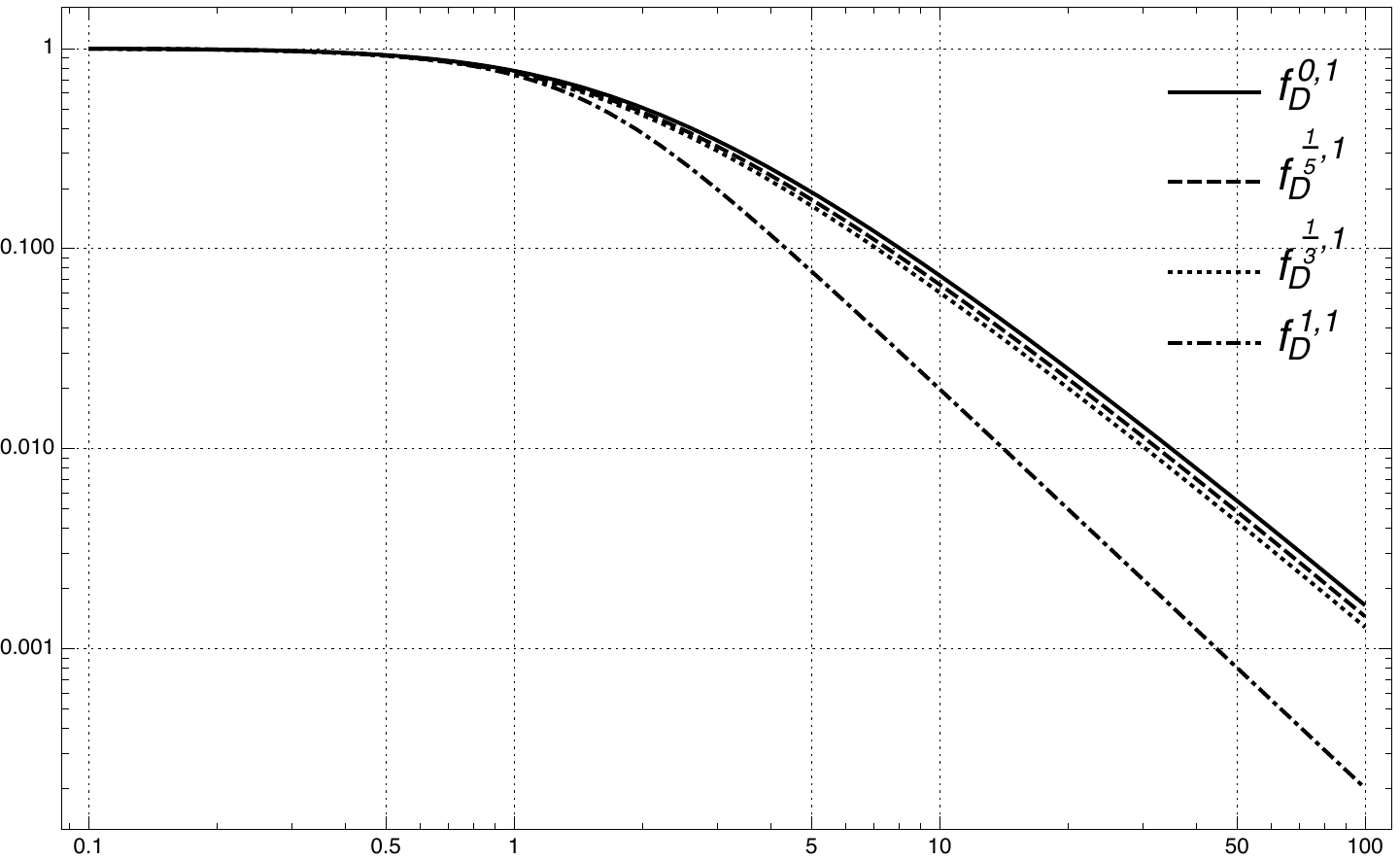}
\par\end{centering}
\caption{\label{fig:Debye-fc-beta1}The Debye function $f_{D}^{\beta,1}=f_{D}(\cdot;\beta,1)$.
Left linear scale; right logaritmic scale.}
\end{figure}

\subsection*{Acknowledgement}

We would like to express our gratitude for the hospitality of our
colleagues and friends Victoria Bernido and Christopher Bernido during
a very pleasant stay at Jagna during the 8th Jagna International Workshop:
``\emph{Structure, Function, and Dynamics: from nm to Gm}\textquotedblright ,
January 4-7, 2017. Financial support from FCT \textendash{} Funda{\c c\~a}o
para a Ci{\^e}ncia e a Tecnologia through the project UID/MAT/04674/2013
(CIMA Universidade da Madeira) and by the Humboldt Foundation are
gratefully acknowledged .


\begin{thebibliography}{GKMS14}

\bibitem[Aga53]{Agarwal1953}
R.\ P.\ Agarwal.
\newblock {A propos d'une note de M.\ Pierre Humbert}.
\newblock {\em CR Acad.\ Sci.\ Paris}, 236(21):2031--2032, 1953.

\bibitem[Bra64]{Braaksma1936}
B.\ L.\ J.\ Braaksma.
\newblock {Asymptotic expansions and analytic continuations for a class of
  Barnes-integrals}.
\newblock {\em Compos.\ Math.}, 15:239--341, 1964.

\bibitem[GJ16]{GJ15}
M.\ Grothaus and F.\ Jahnert.
\newblock {Mittag-Leffler Analysis II: Application to the fractional heat
  equation}.
\newblock {\em J.\ Funct.\ Anal.}, 270(7):2732--2768, April 2016.

\bibitem[GJRS15]{GJRS14}
M.\ Grothaus, F.\ Jahnert, F.\ Riemann, and J.\ L.\ Silva.
\newblock {Mittag-Leffler Analysis I: Construction and characterization}.
\newblock {\em J.\ Funct.\ Anal.}, 268(7):1876--1903, April 2015.

\bibitem[GKMS14]{GKMS2014}
R.\ Gorenflo, A.\ A.\ Kilbas, F.\ Mainardi, and V.\ R.\ Sergei.
\newblock {\em {Mittag-Leffler Functions, Related Topics and Applications}}.
\newblock Springer, 2014.

\bibitem[GV68]{GV68}
I.\ M.\ Gel'fand and N.\ {Ya}.\ Vilenkin.
\newblock {\em Generalized Functions}, volume~4.
\newblock Academic Press, Inc., New York and London, 1968.

\bibitem[Ham16]{Hammouda2008}
B.\ Hammouda.
\newblock Probing nanoscale structures-the sans toolbox.
\newblock Available online at
  \url{http://www.ncnr.nist.gov/staff/hammouda/the_SANS_toolbox.pdf}, 2016.

\bibitem[Hid70]{H70}
T.\ Hida.
\newblock {\em Stationary Stochastic Processes}.
\newblock Princepton University Press, 1970.

\bibitem[KST02]{Kilbas2002}
A.\ A.\ Kilbas, M.\ Saigo, and J.\ J.\ Trujillo.
\newblock {On the generalized Wright function}.
\newblock {\em Fract.\ Calc.\ Appl.\ Anal.}, 5(4):437--460, 2002.

\bibitem[KST06]{KST2006}
A.\ A.\ Kilbas, H.\ M.\ Srivastava, and J.\ J.\ Trujillo.
\newblock {\em {Theory and Applications of Fractional Differential Equations}},
  volume 204 of {\em North-Holland Mathematics Studies}.
\newblock Elsevier Science B.V., Amsterdam, 2006.

\bibitem[Mai10]{Mainardi2010}
F.\ Mainardi.
\newblock {\em {Fractional Calculus and Waves in Linear Viscoelasticity: An
  Introduction to Mathematical Models}}.
\newblock World Scientific, 2010.

\bibitem[MH08]{Mathai2008}
A.\ M.\ Mathai and H.\ J.\ Haubold.
\newblock {\em {Special Functions for Applied Scientists}}.
\newblock Springer, 2008.

\bibitem[ML03]{Mittag-Leffler1903}
G.\ M.\ Mittag-Leffler.
\newblock Sur la nouvelle fonction $e_\alpha (x)$.
\newblock {\em CR Acad.\ Sci.\ Paris}, 137(2):554--558, 1903.

\bibitem[ML04]{Mittag-Leffler1904}
G.\ M.\ Mittag-Leffler.
\newblock Sopra la funzione $e_\alpha(x)$.
\newblock {\em Rend.\ Accad.\ Lincei}, 5(13):3--5, 1904.

\bibitem[ML05]{Mittag-Leffler1905}
G.\ M.\ Mittag-Leffler.
\newblock Sur la repr{\'e}sentation analytique d'une branche uniforme d'une
  fonction monog{\`e}ne.
\newblock {\em Acta Math.}, 29(1):101--181, 1905.

\bibitem[MP07]{MP07}
F.\ Mainardi and G.\ Pagnini.
\newblock {The role of the Fox--Wright functions in fractional sub-diffusion of
  distributed order}.
\newblock {\em J.\ Comput.\ Anal.\ Appl.}, 207:245--257, 2007.

\bibitem[Pol48]{Pollard48}
H.\ Pollard.
\newblock The completely monotonic character of the {M}ittag-{L}effler function
  {$E_a(-x)$}.
\newblock {\em Bull.\ Amer.\ Math.\ Soc.}, 54:1115--1116, 1948.

\bibitem[Sch90]{Schneider90}
W.\ R.\ Schneider.
\newblock Grey noise.
\newblock In S.\ Albeverio, G.\ Casati, U.\ Cattaneo, D.\ Merlini, and R.\ Moresi,
  editors, {\em {Stochastic Processes, Physics and Geometry}}, pages 676--681.
  World Scientific Publishing, Teaneck, NJ, 1990.

\bibitem[Sch92]{MR1190506}
W.\ R.\ Schneider.
\newblock Grey noise.
\newblock In S.\ Albeverio, J.\ E.\ Fenstad, H.\ Holden, and T.\ Lindstr{\o}m,
  editors, {\em {Ideas and Methods in Mathematical Analysis, Stochastics, and
  Applications} ({O}slo, 1988)}, pages 261--282. Cambridge Univ.\ Press,
  Cambridge, 1992.

\bibitem[SKM93]{SKM1993}
S.\ G.\ Samko, A.\ A.\ Kilbas, and O.\ I.\ Marichev.
\newblock {\em Fractional integrals and derivatives}.
\newblock Gordon and Breach Science Publishers, Yverdon, 1993.
\newblock Theory and applications, Edited and with a foreword by S.\ M.\
  Nikol'ski{\u\i}, Translated from the 1987 Russian original, Revised by the
  authors.

\bibitem[Ter02]{Teraoka2002}
I.\ Teraoka.
\newblock {\em {Polymer Solutions: An Introduction to Physical Properties}}.
\newblock New York: Wiley, 2002.

\end{thebibliography}
\end{document}